\documentclass[reqno]{amsart}
\usepackage{amsmath}
\usepackage{amsfonts}

\newtheorem{theorem}{Theorem}[section]
\theoremstyle{plain}

\newtheorem{definition}{Definition}[section]

\newtheorem{lemma}{Lemma}[section]

\newtheorem{remark}{Remark}[section]

\numberwithin{equation}{section}

\begin{document}
\title[Long-time dynamics]{Long-time dynamics of the parabolic $p$-Laplacian
equation}
\author{P. G. Geredeli\ \ }
\address{{\small Department of Mathematics,} {\small Faculty of Science,
Hacettepe University, Beytepe 06800}, {\small Ankara, Turkey}}
\email{pguven@hacettepe.edu.tr}
\author{ \ A. Kh. Khanmamedov}
\email{azer@hacettepe.edu.tr}

\begin{abstract}
In this paper, we study the long-time behaviour of solutions of Cauchy
problem for the parabolic $p$-Laplacian equation with variable coefficients.
Under the mild conditions on the coefficient of the principal part and
without upper growth restriction on the source function, we prove that this
problem possesses a compact and invariant global attractor in $L^{2}(R^{n})$.
\end{abstract}

\maketitle


\section{Introduction}

The main goal of this paper is to discuss the long-time behaviour (in the
terms of attractors) of the solutions for the following equation%
\begin{equation}
u_{t}-div(\sigma (x)\left\vert \nabla u\right\vert ^{p-2}\nabla u)+\beta
(x)u+f(u)=g(x),\text{ \ \ \ }(t,x)\in (0,\infty )\times R^{n},  \tag{1.1}
\end{equation}%
with the initial data%
\begin{equation}
u(0,x)=u_{0}(x),  \tag{1.2}
\end{equation}%
where $p\geq 2$, $g\in L^{2}(R^{n})$, $u_{0}\in L^{2}(R^{n})$, $n\geq 2$.
Here, the functions $\sigma $, $\beta $ and $f$ satisfy the following
assumptions:%
\begin{equation}
\sigma \in L_{loc}^{1}(R^{n}),\ \sigma (\cdot )\geq 0,\  \sigma
^{-\frac{2n}{n(p-2)+2p}}\in L_{loc}^{1}(R^{n}),\ \ \ \ \ \ \ \ \ \
\ \ \ \ \ \ \ \ \ \ \ \ \ \ \ \ \ \ \ \ \ \ \  \tag{1.3}
\end{equation}%
\begin{equation}
\beta \in L^{\infty }(R^{n}),\ \beta (\cdot )\geq 0,\ \beta
(x)\geq \beta _{0}>0\text{ a.e. in }\left\{\left\vert x\right\vert
\geq r_{0}\right\} \text{ for some }r_{0}>0,\tag{1.4}
\end{equation}%
\begin{equation}
f\in C^{1}(R),\text{ }f(s)s\geq 0,\ \forall \text{ }s\in R,\text{
 }f^{\prime }(\cdot )>-c,\text{  }c>0.\text{ \ \ \ \ \ \ \ \ \ \
\ \ \ \ \ \ \ \ \ \ \ \ \ \ \ \ \ \ \ }  \tag{1.5}
\end{equation}

The understanding of the long-time behaviour of dynamical systems is one of
the most important problems of modern mathematics. One way of approaching to
this problem is to analyse the existence of the global attractor. The
existence of the global attractors for the parabolic equations has
extensively been studied by many authors. We refer to [1-7] and the
references therein for the reaction-diffusion equations and to [3, 8-15] for
the evolution $p$-Laplacian equations. When $\sigma (x)\equiv 1$, $\beta
(x)\equiv \lambda ,$ the existence of the global attractor for equation
(1.1) was studied in [3, 8-11] for bounded domains and in [12-16] for
unbounded domains.

In this paper we deal with the equation (1.1) which contains the variable
coefficients $\sigma (\cdot )$ and $\beta (\cdot )$. This type of equations
have recently taken an interest by several authors. In \cite{17}, for the
case $\beta (x)\equiv 0,$ the authors have shown the existence of the global
attractor for equation (1.1) in a bounded domain. In that paper the
diffusion coefficient $\sigma (\cdot )$ is assumed to be like $\left\vert
x\right\vert ^{\alpha }$ for $\alpha \in (0,p)$ and due to the studying in a
bounded domain the authors prove the asymptotic compactness property of the
solutions by using the compact embeddings of Sobolev Spaces. The existence
of the global attractor for equation (1.1), under the assumption
\begin{equation*}
\sigma (x)\sim \left\vert x\right\vert ^{\alpha }+\left\vert x\right\vert
^{\gamma }\text{ , \ }\alpha \in (0,p),\text{ \ }\gamma >p+\frac{n}{2}(p-2),
\end{equation*}%
has been shown in \cite{18}. Although the authors in \cite{18} have studied
the problem in an arbitrary domain, the compact embeddings could also be
used to obtain the asymptotic compactness of solutions because of the
conditions imposed on $\sigma (\cdot ).$

The main novelty in our paper is the following: $(i)$ we weaken the
conditions on the function $\sigma (\cdot )$ which are given in \cite{17}
and \cite{18}, so that the embedding of the space with the norm $\left(
\left\Vert \nabla u\right\Vert _{L_{\sigma }^{p}(\Omega )}+\left\Vert
u\right\Vert _{L^{2}(\Omega )}\right) $ into the space $L^{2}(\Omega )$ is
not compact, for each subdomain $\Omega \subset R^{n}$; \ \ $(ii)$ we remove
the upper growth condition on the source term.

The absence of the upper growth condition on $f$ and the lack of the compact
embedding cause some difficulties for the existence of the solutions and the
asymptotic compactness of the solution operator in $L^{2}(R^{n})$. We prove
the existence of the solutions by Galerkin's method and to overcome the
difficulties related to the limit transition in the source term $f,$ we
apply the weak compactness theorem in the Orlicz spaces. To prove the
asymptotic compactness of the solutions, we first establish the validity of
the energy equalities by using the approximation of the weak solutions by
the bounded functions and then apply the approach of \cite{19} by using the
weak compactness argument.

Our main result is as follows:

\begin{theorem}
Let conditions (1.3)-(1.5) hold. Then problem (1.1)-(1.2) possesses a
compact and invariant global attractor in $L^{2}(R^{n})$.
\end{theorem}

The present paper is organized as follows. In the next section, we give some
definitions and lemmas which will be used in the following sections. In
section 3, the well-posedness of problem (1.1)-(1.2) is proved. In section
4, we show the existence of the absorbing set and present the proof of the
asymptotic compactness to establish our main result.

\section{Preliminaries}

\qquad This section is devoted to give some definitions and lemmas which
will be used in the next sections. In order to study problem (1.1)-(1.2),
let us begin with the introduction of the spaces $W$ and $W_{b}$.

\begin{definition}
Under conditions (1.3)-(1.4), we define the spaces $W$ and $W_{b}$ as the
closure of $C_{0}^{\infty }(R^{n})$ in the following norms respectively,%
\begin{equation*}
\left\Vert u\right\Vert _{W}:=\left\Vert \nabla u\right\Vert _{L_{\sigma
}^{p}(R^{n})}+\left\Vert u\right\Vert _{L_{\beta }^{2}(R^{n})}
\end{equation*}%
\begin{equation*}
=\left( \int\limits_{R^{n}}\sigma (x)\left\vert \nabla u(x)\right\vert
^{p}dx\right) ^{\frac{1}{p}}+\left( \int\limits_{R^{n}}\beta (x)\left\vert
u(x)\right\vert ^{2}dx)\right) ^{\frac{1}{2}},
\end{equation*}%
\begin{equation*}
\left\Vert u\right\Vert _{W_{b}}=\left\Vert u\right\Vert _{W}+\underset{x\in
R^{n}}{\sup }\left\vert u(x)\right\vert .
\end{equation*}
\end{definition}

One can show that $W$ is a separable, reflexive Banach space and $W_{b}$ is
a separable Banach space. Now, before giving the definition of the weak
solution of problem (1.1)-(1.2), let us define the operator $A$ $%
:W\rightarrow W^{\ast }$ as $A\varphi =-div(\sigma (x)\left\vert \nabla
\varphi \right\vert ^{p-2}\nabla \varphi )+\beta (x)\varphi $, where $%
W^{\ast }$ is the dual of $W$. It is easy to show that the operator $%
A:W\rightarrow W^{\ast }$ is bounded, monotone and hemicontinuous.

\begin{definition}
The function $u\in C([0,T];L^{2}(R^{n}))\cap L^{p}(0;T;W)\cap
W_{loc}^{1,2}(0,T;L^{2}(R^{n}))$, satisfying $\int\limits_{0}^{T}\int%
\limits_{R^{n}}f(u(t,x))u(t,x)dxdt<\infty ,\ u(0,x)=u_{0}(x)$ and the
equation%
\begin{equation*}
\left\langle u_{t},v\right\rangle +\left\langle Au,v\right\rangle
+\left\langle f(u),v\right\rangle =\left\langle g,v\right\rangle \text{ \
a.e. \ on }(0,T)\text{,}
\end{equation*}%
for all $v\in W\cap L^{\infty }(R^{n})$, is called the weak solution to
problem (1.1)-(1.2), where $\left\langle \cdot ,\cdot \right\rangle $ is the
dual form between $W$ and $W^{\ast }$.
\end{definition}

\begin{remark}
To give a meaning to the third term on the left hand side of the equality
given in the definition, it is enough to see that $f(u)$ $\in
L^{1}(0,T;L^{1}(R^{n}))+L^{2}(0,T;L^{2}(R^{n}))$. Let $\chi _{\Omega _{1}}$
and $\chi _{\Omega _{2}}$ be the characteristic functions of the sets
\begin{equation*}
\Omega _{1}=\{(t,x)\in (0,T)\times R^{n}:\left\vert u(t,x)\right\vert >1\},
\end{equation*}%
\begin{equation*}
\Omega _{2}=\{(t,x)\in (0,T)\times R^{n}:\left\vert u(t,x)\right\vert \leq
1\}.
\end{equation*}%
Since%
\begin{equation*}
\int\limits_{0}^{T}\int\limits_{R^{n}}\left\vert f(u(t,x))\chi _{\Omega
_{1}}(t,x)\right\vert dxdt=\iint\limits_{\Omega _{1}}\left\vert
f(u(t,x))\right\vert dxdt
\end{equation*}%
\begin{equation*}
\leq \iint\limits_{\Omega _{1}}f(u(t,x))u(t,x)dxdt\leq
\int\limits_{0}^{T}\int\limits_{R^{n}}f(u(t,x))u(t,x)dxdt<\infty ,
\end{equation*}%
we get $f(u)\chi _{\Omega _{1}}\in $ $L^{1}(0,T;L^{1}(R^{n}))$ . On the
other hand, since
\begin{equation*}
\int\limits_{0}^{T}\int\limits_{R^{n}}\left\vert f(u(t,x))\chi _{\Omega
_{2}}(t,x)\right\vert ^{2}dxdt=\iint\limits_{\Omega _{2}}\left\vert
f(u(t,x))\right\vert ^{2}dxdt
\end{equation*}%
\begin{equation*}
=\iint\limits_{\Omega _{2}}\left\vert \left( f(u(t,x))-f(0)\right)
\right\vert ^{2}dxdt,
\end{equation*}%
by Mean Value theorem we have%
\begin{equation*}
\int\limits_{0}^{T}\int\limits_{R^{n}}\left\vert f(u(t,x))\chi _{\Omega
_{2}}(t,x)\right\vert ^{2}dxdt\leq C\iint\limits_{\Omega _{2}}\left\vert
u(t,x)\right\vert ^{2}dxdt\leq
C\int\limits_{0}^{T}\int\limits_{R^{n}}\left\vert u(t,x)\right\vert ^{2}dxdt.
\end{equation*}%
Taking into account that $u\in L^{\infty }(0;T;L^{2}(R^{n}))$, we get $%
f(u)\chi _{\Omega _{2}}\in L^{2}(0,T;L^{2}(R^{n}))$. Since%
\begin{equation*}
f(u(t,x))=f(u(t,x))\chi _{\Omega _{1}}(t,x)+f(u(t,x))\chi _{\Omega
_{2}}(t,x).
\end{equation*}%
we have $f(u)$ $\in L^{1}(0,T;L^{1}(R^{n}))+L^{2}(0,T;L^{2}(R^{n})).$
\end{remark}

\begin{lemma}
The inequality
\begin{equation*}
\left\Vert \nabla u\right\Vert _{L^{\frac{2n}{n+2}}(B(0,2r))}^{2}+\left\Vert
u\right\Vert _{L^{2}(R^{n}\backslash B(0,r))}^{2}\geq C\left\Vert
u\right\Vert _{L^{2}(R^{n})}^{2}
\end{equation*}%
is satisfied for all $u\in C_{0}^{\infty }(R^{n})$ and $r>0,$ where $%
B(0,r)=\{x:x\in R^{n},\ \ |x|<r\}$ and the positive constant $C$ depends on $%
r$ and $n$.
\end{lemma}

\begin{proof}
Let $\varphi (\cdot )\in C_{0}^{\infty }(R^{n})$ be such that $0\leq \varphi
(x)\leq 1$ and%
\begin{equation*}
\varphi (x)=\left\{
\begin{array}{c}
1,\text{ \ \ \ \ \ }x\in B(0,1), \\
\text{\ }0,\text{ \ }x\in R^{n}\backslash B(0,2),%
\end{array}%
\right.
\end{equation*}%
furthermore define $\varphi _{r}(x)=\varphi (\frac{x}{r})$ and $\widetilde{u}%
(x)=(u\varphi _{r})(x)$. By the Sobolev inequality we have%
\begin{equation}
\left\Vert \widetilde{u}\right\Vert _{L^{2}(R^{n})}^{2}\leq c_{1}\left\Vert
\nabla \widetilde{u}\right\Vert _{L^{\frac{2n}{n+2}}(R^{n})}^{2}.  \tag{2.1}
\end{equation}%
Now, since%
\begin{equation}
\left\Vert u\right\Vert _{L^{2}(B(0,r))}^{2}=\left\Vert \widetilde{u}%
\right\Vert _{L^{2}(B(0,r))}^{2}\leq \left\Vert \widetilde{u}\right\Vert
_{L^{2}(R^{n})}^{2},  \tag{2.2}
\end{equation}%
and
\begin{equation*}
\left( \int\limits_{R^{n}}\left\vert \nabla \widetilde{u}(x)\right\vert ^{%
\frac{2n}{n+2}}dx\right) ^{\frac{n+2}{n}}=\left(
\int\limits_{R^{n}}\left\vert \nabla (u\varphi _{r})(x)\right\vert ^{\frac{2n%
}{n+2}}dx\right) ^{\frac{n+2}{n}}
\end{equation*}%
\begin{equation*}
\leq c_{2}\left( \left( \int\limits_{R^{n}}\left\vert \nabla u(x)\left\vert
^{\frac{2n}{n+2}}\right\vert \varphi _{r}(x)\right\vert ^{\frac{2n}{n+2}%
}dx\right) ^{\frac{n+2}{n}}+\frac{1}{r^{2}}\left(
\int\limits_{R^{n}}\left\vert u\left\vert ^{\frac{2n}{n+2}}\right\vert
(\nabla \varphi )(\frac{x}{r})\right\vert ^{\frac{2n}{n+2}}dx\right) ^{\frac{%
n+2}{n}}\right)
\end{equation*}%
\begin{equation}
\leq c_{3}\left( \left\Vert \nabla u\right\Vert _{L^{\frac{2n}{n+2}%
}(B(0,2r))}^{2}+\left\Vert u\right\Vert _{L^{2}(B(0,2r)\backslash
B(0,r))}^{2}\right) ,  \tag{2.3}
\end{equation}%
from (2.1)-(2.3), we obtain%
\begin{equation*}
\left\Vert u\right\Vert _{L^{2}(B(0,r))}^{2}\leq c_{4}\left( \left\Vert
\nabla u\right\Vert _{L^{\frac{2n}{n+2}}(B(0,2r))}^{2}+\left\Vert
u\right\Vert _{L^{2}(B(0,2r)\backslash B(0,r))}^{2}\right) .
\end{equation*}%
By adding $\left\Vert u\right\Vert _{L^{2}(R^{n}\backslash B(0,r))}^{2}$ to
the both sides of the above inequality we get the claim of the lemma.
\end{proof}

\begin{lemma}
Assume that the conditions (1.3)-(1.4) are satisfied. Then for all $u\in
C_{0}^{\infty }(R^{n})$ the inequality
\begin{equation*}
\left\Vert u\right\Vert _{L^{2}(R^{n})}\leq \overline{C}\left\Vert
u\right\Vert _{W},
\end{equation*}%
which yields $W\subset L^{2}(R^{n})$, is satisfied.
\end{lemma}

\begin{proof}
The Holder inequality yields
\begin{equation*}
\int\limits_{B(0,r)}\left\vert \nabla u(x)\right\vert ^{\frac{2n}{n+2}%
}dx\leq \left( \int\limits_{B(0,r)}\sigma (x)\left\vert \nabla
u(x)\right\vert ^{p}dx\right) ^{\frac{2n}{p(n+2)}}\left(
\int\limits_{B(0,r)}\sigma ^{-\frac{2n}{p(n+2)-2n}}(x)dx\right) ^{\frac{%
p(n+2)-2n}{p(n+2)}},
\end{equation*}%
for every $r>0$. From the assumption (1.3) it follows that%
\begin{equation}
\left\Vert \nabla u\right\Vert _{L^{\frac{2n}{n+2}}(B(0,r))}\leq
c(r)\left\Vert u\right\Vert _{W},  \tag{2.4}
\end{equation}%
On the other hand by condition (1.4), we have
\begin{equation}
\left\Vert u\right\Vert _{L^{2}(R^{n}\backslash B(0,r_{0}))}\leq \beta_{0}^{-%
\frac{1}{2}}\left\Vert u\right\Vert _{W}.  \tag{2.5}
\end{equation}%
Taking into account the previous lemma and (2.4) - (2.5), we obtain the
result.
\end{proof}

\begin{lemma}
Let the conditions (1.3)-(1.4) are satisfied. Then $B_{k}:W\rightarrow W$ is
a continuous map and $\underset{k\rightarrow \infty }{\lim }\left\Vert
u-B_{k}(u)\right\Vert _{W}=0$, for every $u\in W$, where $B_{k}(s)=\left\{
\begin{array}{c}
k,\text{ \ \ }s>k, \\
s,\text{ \ }\left\vert s\right\vert \leq k, \\
-k,\text{ \ }s<-k,%
\end{array}%
\right. $ for $s\in R$.
\end{lemma}

\begin{proof}
By the definition of $W,$ for any $u\in W$ there exists a sequence $\left\{
u_{m}\right\} _{m=1}^{\infty }\subset C_{0}^{\infty }(R^{n})$ such that%
\begin{equation}
\underset{m\rightarrow \infty }{\lim }\left\Vert u-u_{m}\right\Vert _{W}=0,
\tag{2.6}
\end{equation}%
which according to the Lemma 2.2 yields%
\begin{equation}
\underset{m\rightarrow \infty }{\lim }\left\Vert u-u_{m}\right\Vert
_{L^{2}(R^{n})}=0.  \tag{2.7}
\end{equation}%
Now, let us show that
\begin{equation}
\underset{m\rightarrow \infty }{\lim }\left\Vert
B_{k}(u)-B_{k}(u_{m})\right\Vert _{W}=0.  \tag{2.8}
\end{equation}%
Since
\begin{equation*}
\left\vert B_{k}(u)-B_{k}(v)\right\vert \leq \left\vert u-v\right\vert ,%
\text{ \ \ \ }\forall u,v\in R,
\end{equation*}%
by (1.4) and (2.7), we have%
\begin{equation}
\underset{m\rightarrow \infty }{\lim \sup }\int\limits_{R^{n}}\beta
(x)\left\vert B_{k}(u)(x)-B_{k}(u_{m})(x)\right\vert ^{2}dx\leq \left\Vert
\beta \right\Vert _{L^{\infty }(R^{n})}\underset{m\rightarrow \infty }{\lim
\sup }\left\Vert u-u_{m}\right\Vert _{L^{2}(R^{n})}^{2}=0.  \tag{2.9}
\end{equation}%
By (2.6), it follows that
\begin{equation}
\underset{m\rightarrow \infty }{\lim }\int\limits_{E}\sigma (x)\left\vert
\nabla u_{m}(x)\right\vert ^{p}dx=\int\limits_{E}\sigma (x)\left\vert \nabla
u(x)\right\vert ^{p}dx,  \tag{2.10}
\end{equation}%
for every measurable $E\subset R^{n}$. By the last equality, we find%
\begin{equation*}
\underset{r\rightarrow \infty }{\lim \sup }\underset{m\rightarrow \infty }{%
\lim \sup }\int\limits_{R^{n}\backslash B(0,r)}\sigma (x)\left\vert \nabla
B_{k}(u)(x)-\nabla B_{k}(u_{m})(x)\right\vert ^{p}dx
\end{equation*}%
\begin{equation*}
\leq 2^{p}\underset{r\rightarrow \infty }{\lim \sup }\underset{m\rightarrow
\infty }{\lim \sup }\left( \int\limits_{R^{n}\backslash B(0,r)}\sigma
(x)\left\vert \nabla u_{m}\right\vert ^{p}dx+\int\limits_{R^{n}\backslash
B(0,r)}\sigma (x)\left\vert \nabla u\right\vert ^{p}dx\right)
\end{equation*}%
\begin{equation}
=2^{p+1}\underset{r\rightarrow \infty }{\lim \sup }\int\limits_{R^{n}%
\backslash B(0,r)}\sigma (x)\left\vert \nabla u\right\vert ^{p}dx=0.
\tag{2.11}
\end{equation}%
By (2.4) and (2.6), there exists a subsequence $\left\{ u_{m_{j}}\right\}
_{j=1}^{\infty }$ such that
\begin{equation*}
\nabla u_{m_{j}}(x)\rightarrow \nabla u(x)\text{ \ a.e. on \ }B(0,r).
\end{equation*}%
Then by Egorov's theorem for any $\delta >0$ there exists a measurable $%
E_{\delta }\subset B(0,r)$ such that $mes(E_{\delta })<\delta $ and
\begin{equation}
\nabla u_{m_{j}}(x)\rightarrow \nabla u(x)\text{ \ uniformly on }%
B(0,r)\backslash E_{\delta }.  \tag{2.12}
\end{equation}%
By (2.10) and (2.12), we get%
\begin{equation*}
\underset{j\rightarrow \infty }{\lim \sup }\int\limits_{B(0,r)}\sigma
(x)\left\vert \nabla B_{k}(u)(x)-\nabla B_{k}(u_{m_{j}})(x)\right\vert
^{p}dx\leq 2^{p+1}\underset{\delta \rightarrow 0}{\lim \sup }%
\int\limits_{E_{\delta }}\sigma (x)\left\vert \nabla u(x)\right\vert ^{p}dx
\end{equation*}%
\begin{equation*}
+\underset{\delta \rightarrow 0}{\lim \sup }\underset{j\rightarrow \infty }{%
\lim \sup }\int\limits_{\{x:\text{ }k-\delta \leq \left\vert u(x)\right\vert
\leq k+\delta \}\cap (B(0,r)\backslash E_{\delta })}\sigma (x)\left\vert
\nabla B_{k}(u)(x)-\nabla B_{k}(u_{m_{j}})\right\vert ^{p}dx
\end{equation*}%
\begin{equation*}
\leq 2^{p+1}\underset{\delta \rightarrow 0}{\lim \sup }\int\limits_{\{x:%
\text{ }k-\delta \leq \left\vert u(x)\right\vert \leq k+\delta \}\cap
B(0,r)}\sigma (x)\left\vert \nabla u(x)\right\vert ^{p}dx
\end{equation*}%
\begin{equation*}
=2^{p+1}\int\limits_{\{x:\text{ }\left\vert u(x)\right\vert =k\}\cap
B(0,r)}\sigma (x)\left\vert \nabla u(x)\right\vert ^{p}dx
\end{equation*}%
\begin{equation*}
\leq 2^{p+1}\int\limits_{\{x:\text{ }u(x)=k\}}\sigma (x)\left\vert \nabla
u(x)\right\vert ^{p}dx+2^{p+1}\int\limits_{\{x:\text{ }u(x)=-k\}}\sigma
(x)\left\vert \nabla u(x)\right\vert ^{p}dx=0.
\end{equation*}%
By the same way, one can show that every subsequence of $\left\{
u_{m}\right\} _{m=1}^{\infty }$ has a subsequence satisfying the above
equality. So, we have
\begin{equation}
\underset{m\rightarrow \infty }{\lim }\int\limits_{B(0,r)}\sigma
(x)\left\vert \nabla B_{k}(u)(x)-\nabla B_{k}(u_{m})(x)\right\vert ^{p}dx=0.
\tag{2.13}
\end{equation}%
By (2.9), (2.11) and (2.13), we obtain (2.8).

Now, let us show that for any $u\in C_{0}^{\infty }(R^{n})$ and $k\in
\mathbb{N}
$ there exists $\left\{ v_{m}\right\} _{m=1}^{\infty }\subset C_{0}^{\infty
}(R^{n})$ such that%
\begin{equation}
\underset{m\rightarrow \infty }{\lim }\left\Vert B_{k}(u)-v_{m}\right\Vert
_{W}=0.  \tag{2.14}
\end{equation}%
Denote $v_{m}(x)=(\rho _{m}\ast B_{k}(u))(x)$, where $\rho _{m}(x)=\left\{
\begin{array}{c}
Km^{n}e^{-\frac{1}{1-m^{2}\left\vert x\right\vert ^{2}}},\text{ \ }%
\left\vert x\right\vert <\frac{1}{m}, \\
0,\text{ \ \ \ \ \ \ \ \ \ \ \ \ \ \ \ }\left\vert x\right\vert \geq \frac{1%
}{m}%
\end{array}%
\right. $, $m\in
\mathbb{N}
$ and $K^{-1}=\int_{\left\{ x:\text{ }\left\vert x\right\vert <1\right\}
}e^{-\frac{1}{1-\left\vert x\right\vert ^{2}}}dx$. Since $u\in C_{0}^{\infty
}(R^{n})$, by the definition of the function $B_{k}(\cdot )$, we have $%
B_{k}(u)\in W_{0}^{1,\infty }(B(0,r))\cap C_{0}(B(0,r))$ and $v_{m}\in
C_{0}^{\infty }(B(0,r+1))$, for some $r>0$. It is well known that
\begin{equation}
\underset{m\rightarrow \infty }{\lim }\left\Vert B_{k}(u)-v_{m}\right\Vert
_{W^{1,2}(R^{n})}=0.  \tag{2.15}
\end{equation}%
By (1.4) and (2.15), we obtain%
\begin{equation}
\underset{m\rightarrow \infty }{\lim \sup }\int\limits_{R^{n}}\beta
(x)\left\vert v_{m}(x)-B_{k}(u)(x)\right\vert ^{2}dx=0.  \tag{2.16}
\end{equation}%
Also by (2.15), we have
\begin{equation*}
\nabla v_{m}\rightarrow \nabla B_{k}(u)\text{ in measure on }R^{n}\text{.}
\end{equation*}%
Since
\begin{equation*}
\left\Vert \nabla v_{m}\right\Vert _{L^{\infty }(R^{n})}\leq \left\Vert
\nabla B_{k}(u)\right\Vert _{L^{\infty }(R^{n})}\leq c,\text{ \ \ \ }\forall
m,k\in
\mathbb{N}
,
\end{equation*}%
applying Lebesgue's convergence theorem, we get%
\begin{equation*}
\underset{m\rightarrow \infty }{\lim }\int\limits_{R^{n}}\sigma
(x)\left\vert \nabla v_{m}(x)-\nabla B_{k}(u)(x)\right\vert ^{p}dx
\end{equation*}%
\begin{equation*}
=\underset{m\rightarrow \infty }{\lim }\int\limits_{B(0,r+1)}\sigma
(x)\left\vert \nabla v_{m}(x)-\nabla B_{k}(u)(x)\right\vert ^{p}dx=0.
\end{equation*}%
which together with (2.16) yields (2.14). By (2.8) and (2.14), for every $%
u\in W$ and $k\in
\mathbb{N}
$ there exists a sequence $\left\{ u_{m}\right\} _{m=1}^{\infty }\subset
C_{0}^{\infty }(R^{n})$ converging to $B_{k}(u)$ in the norm of $W$. It
means that $B_{k}(u)\in W$, for every $u\in W$.

Now, using the argument done in the proof of (2.8), one can prove that if $%
\left\{ u_{m}\right\} _{m=1}^{\infty }$ converges to $u$ in $W$ as $%
m\rightarrow \infty $, then $\left\{ B_{k}(u_{m})\right\} _{m=1}^{\infty }$
converges to $B_{k}(u)$ in $W$ as $m\rightarrow \infty $. Also, by the
definition of the function $B_{k}(\cdot )$, it is easy to show that $\left\{
B_{k}(u)\right\} _{k=1}^{\infty }$ converges to $u$ in $W$ as $k\rightarrow
\infty $.
\end{proof}

\begin{remark}
By (2.8) and (2.14), for every $u\in W$ and $k\in
\mathbb{N}
$ there exists a sequence $\left\{ v_{m}\right\} _{m=1}^{\infty }\subset
C_{0}^{\infty }(R^{n})$ such that
\begin{equation*}
\underset{m\rightarrow \infty }{\lim }\left\Vert B_{k}(u)-v_{m}\right\Vert
_{W}=0\text{ \ and \ }\underset{m}{\sup }\left\Vert v_{m}\right\Vert
_{L^{\infty }(R^{n})}\leq k.
\end{equation*}%
On the other hand, by the definition of $B_{k}(\cdot )$, for every $u\in
L^{\infty }(R^{n})$ and $k\geq \left\Vert u\right\Vert _{L^{\infty }(R^{n})}$%
\begin{equation*}
B_{k}(u)=u\text{ \ a.e. on \ }R^{n}.
\end{equation*}%
Hence, for every \ $u\in W\cap $ $L^{\infty }(R^{n})$ there exists a
sequence $\left\{ w_{m}\right\} _{m=1}^{\infty }\subset C_{0}^{\infty
}(R^{n})$ such that%
\begin{equation*}
\underset{m\rightarrow \infty }{\lim }\left\Vert u-w_{m}\right\Vert _{W}=0%
\text{ \ and \ }\underset{m}{\sup }\left\Vert w_{m}\right\Vert _{L^{\infty
}(R^{n})}\leq \left\Vert u\right\Vert _{L^{\infty }(R^{n})}+1.
\end{equation*}
\end{remark}

\section{Well-posedness}

We prove the existence of the weak solution to problem (1.1)-(1.2) by
Galerkin's method.

\begin{theorem}
Assume that the conditions (1.3)-(1.5) are satisfied. Then for any $u_{0}\in
L^{2}(R^{n})$ and $T>0,$ there exists a weak solution to (1.1)-(1.2).
\end{theorem}

\begin{proof}
Let us consider the approximate solutions $\{u_{m}(t)\}_{m=1}^{\infty }$ in
the form%
\begin{equation*}
u_{m}(t)=\sum\limits_{k=1}^{m}c_{mk}(t)e_{k},
\end{equation*}%
where $\{e_{j}\}_{j=1}^{\infty }$ $\subset C_{0}^{\infty }(R^{n})$ is a
basis of the space $W_{b}$ and the functions $\{c_{mk}(t)\}_{k=1}^{m}$ are
the solutions of the following problem :%
\begin{equation}
\left\{
\begin{array}{c}
\left\langle \sum\limits_{k=1}^{m}c_{mk}^{\prime
}(t)e_{k},e_{j}\right\rangle +\left\langle
A(\sum\limits_{k=1}^{m}c_{mk}(t)e_{k}),e_{j}\right\rangle +\left\langle
f(\sum\limits_{k=1}^{m}c_{mk}(t)e_{k}),e_{j}\right\rangle \\
=\left\langle g,e_{j}\right\rangle \text{, \ }t>0\text{, \ \ }j=1,...,m,%
\text{ \ \ \ \ \ \ \ \ \ \ \ \ \ \ \ \ \ \ \ \ \ \ \ \ \ \ \ \ \ \ \ \ \ \ \
\ \ \ \ \ \ \ \ \ \ \ \ \ \ } \\
\sum\limits_{k=1}^{m}c_{mk}(0)e_{k}\rightarrow u_{0}\text{ strongly in }%
L^{2}(R^{n})\text{ as }m\rightarrow \infty .\text{ \ \ \ \ \ \ \ \ \ \ \ \ \
\ \ \ \ \ \ \ \ \ \ \ }%
\end{array}%
\right.  \tag{3.1}
\end{equation}%
By the boundedness, monotonicity and hemicontinuity of $A:W\rightarrow
W^{\ast }$ it follows that this operator is demicontinuous (see \cite[Lemma
2.1 and Lemma 2.2, p. 38]{20}). So, since $det(\left\langle
e_{j},e_{k}\right\rangle )\neq 0$ and $f$ is continuous, by the Peano
existence theorem, there exists at least one local solution to (3.1) in the
interval $[0,T_{m})$. Multiplying the equation (3.1)$_{j}$, by the function $%
c_{mj}(t)$, for each $j$, adding these relations for \ $j=1,...,m$ and
integrating over $(0,t),$ we have%
\begin{equation*}
\left\Vert u_{m}(t)\right\Vert
_{L^{2}(R^{n})}^{2}+2\int\limits_{0}^{t}\int\limits_{R^{n}}\left( \sigma
(x)\left\vert \nabla u_{m}(\tau ,x)\right\vert ^{p}+\beta (x)\left\vert
u_{m}(\tau ,x)\right\vert ^{2}\right) dxd\tau
\end{equation*}%
\begin{equation*}
+2\int\limits_{0}^{t}\int\limits_{R^{n}}f(u_{m}(\tau ,x))u_{m}(\tau
,x)dxd\tau
\end{equation*}%
\begin{equation}
=2\int\limits_{0}^{t}\int\limits_{R^{n}}g(x)u_{m}(\tau ,x)dxd\tau
+\left\Vert u_{m}(0)\right\Vert _{L^{2}(R^{n})}^{2},\text{ \ \ \ }0\leq
t<T_{m}.  \tag{3.2}
\end{equation}%
Since by the last equality
\begin{equation}
\left\Vert u_{m}\right\Vert _{L^{\infty }(0,T_{m};L^{2}(R^{n}))}\leq c_{1},
\tag{3.3}
\end{equation}%
we can extend the approximate solution to the interval $[0,T]$, for every $%
T>0$. Taking into account (3.3) in (3.2), we get%
\begin{equation*}
\left\Vert u_{m}(t)\right\Vert
_{L^{2}(R^{n})}^{2}+\int\limits_{0}^{t}\int\limits_{R^{n}}\left( \sigma
(x)\left\vert \nabla u_{m}(\tau ,x)\right\vert ^{p}+\beta (x)\left\vert
u_{m}(\tau ,x)\right\vert ^{2}\right) dxd\tau
\end{equation*}%
\begin{equation}
+\int\limits_{0}^{t}\int\limits_{R^{n}}\widetilde{f}(u_{m}(\tau
,x))u_{m}(\tau ,x)dxd\tau \leq c_{2},\text{ \ }\forall t\in \lbrack 0,T],
\tag{3.4}
\end{equation}%
where $\widetilde{f}(u_{m}(t,x))=f(u_{m}(t,x))+cu_{m}(t,x)$ and $c$ is the
constant in condition (1.5). \ Now, multiplying equation (3.1)$_{j}$ by the
function $c_{mj}^{\prime }(t)$, for each $j$, adding these relations for \ $%
j=1,...,m$, integrating over $(s,T)$ and taking into account (3.3), we have%
\begin{equation*}
\int\limits_{s}^{T}\left\Vert u_{mt}(t)\right\Vert
_{L^{2}(R^{n})}^{2}dt+\int\limits_{R^{n}}F(u_{m}(T,x))dx\leq
\int\limits_{R^{n}}\sigma (x)\left\vert \nabla u_{m}(s,x)\right\vert ^{p}dx
\end{equation*}%
\begin{equation*}
+\int\limits_{R^{n}}\beta (x)\left\vert u_{m}(s,x)\right\vert
^{2}dx+\int\limits_{R^{n}}F(u_{m}(s,x))dx+c_{3},
\end{equation*}%
where $F(u)=\int\limits_{0}^{u}f(s)ds$. Integrating the last inequality over
$(0,T)$ with respect to the variable $s$ and taking into account (3.4), we
get%
\begin{equation}
\left\Vert u_{mt}\right\Vert _{L^{2}(\varepsilon ,T;L^{2}(R^{n}))}\leq
c_{4}(\varepsilon ),  \tag{3.5}
\end{equation}%
for every $\varepsilon \in (0,T).$ By the estimates (3.3)-(3.5) and the
boundedness of the operator $A:L^{p}(0,T;W)\rightarrow L^{\frac{p}{p-1}%
}(0,T;W^{\ast })$, we obtain (up to a subsequence) that%
\begin{equation}
\left\{
\begin{array}{c}
u_{m}\rightarrow u\text{ \ weakly star in }L^{\infty }(0,T;L^{2}(R^{n})),%
\text{ \ \ \ \ \ \ \ \ \ \ \ } \\
u_{mt}\rightarrow u_{t}\text{ \ weakly in }L^{2}(\varepsilon
,T;L^{2}(R^{n})),\text{ \ }\forall \varepsilon \in (0,T), \\
u_{m}\rightarrow u\text{ \ weakly\ in }L^{p}(0,T;W),\text{ \ \ \ \ \ \ \ \ \
\ \ \ \ \ \ \ \ \ \ \ \ \ \ \ } \\
Au_{m}\rightarrow \chi \text{ \ weakly\ in }L^{\frac{p}{p-1}}(0,T;W^{\ast }),%
\text{ \ \ \ \ \ \ \ \ \ \ \ \ \ \ \ \ \ }%
\end{array}%
\right.  \tag{3.6}
\end{equation}%
as $m\rightarrow \infty $, for some $\chi \in $\ $L^{\frac{p}{p-1}%
}(0,T;W^{\ast })$. So, by (3.6), we get
\begin{equation}
u\in C([\varepsilon ,T];L^{2}(R^{n})).  \tag{3.7}
\end{equation}%
Since by (2.4) and (3.4), the sequence $\{u_{m}\}$ is bounded in $%
L^{p}(0,T;W_{loc}^{1,\frac{2n}{n+2}}(R^{n}))$, using (3.5) and Aubin type
compact embedding theorem (see \cite[Corollary 4]{21}), we have the
compactness of $\{u_{m}\}$ in $L^{1}(\varepsilon ,T;L_{loc}^{1}(R^{n}))$ for
every $\varepsilon \in (0,T).$ Hence there exists subsequences $%
\{u_{m_{n}}^{(k)}\}_{n=1}^{\infty }\subset
\{u_{m_{n}}^{(k-1)}\}_{n=1}^{\infty }\subset ...\subset
\{u_{m}\}_{m=1}^{\infty }$ and $\varepsilon _{k}\searrow 0$ such that%
\begin{equation*}
u_{m_{n}}^{(k)}\rightarrow u\text{ \ \ a.e. \ on \ }(\varepsilon
_{k},T)\times B(0,k),
\end{equation*}%
as $n\rightarrow \infty .$ Now, applying the diagonalization procedure, we
obtain (up to a subsequence $\{u_{m_{k}}^{(k)}\}$) that
\begin{equation}
u_{m}\rightarrow u\text{ \ \ a.e. \ on \ }(0,T)\times R^{n},  \tag{3.8}
\end{equation}%
as $m\rightarrow \infty .$ Now, \ because the sign of the function $%
\widetilde{f}(u)$ is the same as the sign of $u,$ together with (3.4), it
follows that%
\begin{equation}
\left\{
\begin{array}{c}
\int\limits_{0}^{t}\int\limits_{R^{n}}\widetilde{f}(u_{m}^{+}(\tau
,x))u_{m}^{+}(\tau ,x)dxd\tau \leq c_{5}, \\
\int\limits_{0}^{t}\int\limits_{R^{n}}\widetilde{f}(u_{m}^{-}(\tau
,x))u_{m}^{-}(\tau ,x)dxd\tau \leq c_{5}%
\end{array}%
\right. ,\text{ \ }\forall t\in \lbrack 0,T],  \tag{3.9}
\end{equation}%
where $u_{m}^{+}=\max \{u_{m},0\}$ and $u_{m}^{-}=\min \{u_{m},0\}.$ Since
the function $\widetilde{f}(\cdot )$ is continuous, increasing, positive for
$x>0$ and $\widetilde{f}(0)=0$, we can define an $N$-function (see \cite{22}
for definition)
\begin{equation*}
\widetilde{F}(x)=\int\limits_{0}^{\left\vert x\right\vert }\widetilde{f}%
(s)ds,
\end{equation*}%
which has a complementary $N$-function $\widetilde{G}$ as follows:%
\begin{equation*}
\widetilde{G}(y)=\int\limits_{0}^{\left\vert y\right\vert }\widetilde{f}%
^{-1}(\tau )d\tau .
\end{equation*}%
By definition of $\widetilde{G}(\cdot )$ and (3.9)$_{1}$, we get%
\begin{equation*}
\int\limits_{0}^{T}\int\limits_{R^{n}}\widetilde{G}(\widetilde{f}%
(u_{m}^{+}(\tau ,x))dxd\tau \leq \int\limits_{0}^{T}\int\limits_{R^{n}}%
\widetilde{f}(u_{m}^{+}(\tau ,x))u_{m}^{+}(\tau ,x)dxd\tau \leq c_{5},
\end{equation*}%
and consequently we obtain
\begin{equation}
\left\Vert \widetilde{f}(u_{m}^{+})\right\Vert _{L_{\widetilde{G}}^{\ast
}((0,T)\times B(0,k)\text{ })}\leq c_{5}+1,  \tag{3.10}
\end{equation}%
for every $k\in \mathbb{N}$, where $L_{\widetilde{G}}^{\ast }((0,T)\times
B(0,k))$ is the Orlicz space (see \cite{22} for definition). On the other
hand, defining $g(s)=-\widetilde{f}^{-1}(-s)$ for $s>0,$ we can construct a
new $N$-function $\Phi $ such as%
\begin{equation*}
\Phi (y)=\int\limits_{0}^{\left\vert y\right\vert }g(\xi )d\xi .
\end{equation*}%
Choosing $y=-\widetilde{f}(u_{m}^{-})$ and taking into account (3.9)$_{2}$,
we get%
\begin{equation*}
\int\limits_{0}^{T}\int\limits_{R^{n}}\Phi (\widetilde{f}(u_{m}^{-}(\tau
,x))dxd\tau \leq \int\limits_{0}^{T}\int\limits_{R^{n}}\widetilde{f}%
(u_{m}^{-}(\tau ,x))u_{m}^{-}(\tau ,x)dxd\tau \leq c_{5},
\end{equation*}%
and consequently%
\begin{equation}
\left\Vert \widetilde{f}(u_{m}^{-})\right\Vert _{L_{\Phi }^{\ast
}((0,T)\times B(0,k)\text{ })}\leq c_{5}+1,  \tag{3.11}
\end{equation}%
for every $k\in \mathbb{N}$. By using (3.8), continuity of $\widetilde{f}%
(\cdot )$ and the functions $\max \{s,0\}$ and $\min \{s,0\},$ it can be
inferred that%
\begin{equation}
\left\{
\begin{array}{c}
\widetilde{f}(u_{m}^{+})\rightarrow \widetilde{f}(u^{+})\text{ \ in measure
on\ }(0,T)\times B(0,k), \\
\widetilde{f}(u_{m}^{-})\rightarrow \widetilde{f}(u^{-})\text{ \ in measure
on\ }(0,T)\times B(0,k).%
\end{array}%
\right.  \tag{3.12}
\end{equation}%
Now, taking into account (3.10)-(3.12) and using the \textbf{\cite[Theorem
14.6, p. \textbf{132}]{22}}, we get%
\begin{equation*}
\int\limits_{0}^{T}\int\limits_{B(0,k)\text{ }}\widetilde{f}%
(u_{m}^{+}(t,x))v(t,x)dxdt\rightarrow \int\limits_{0}^{T}\int\limits_{B(0,k)%
\text{ }}\widetilde{f}(u^{+}(t,x))v(t,x)dxdt,\text{ \ \ }\forall \text{ }%
v\in E_{\widetilde{F}},
\end{equation*}%
\begin{equation*}
\int\limits_{0}^{T}\int\limits_{B(0,k)\text{ }}\widetilde{f}%
(u_{m}^{-}(t,x))w(t,x)dxdt\rightarrow \int\limits_{0}^{T}\int\limits_{B(0,k)%
\text{ }}\widetilde{f}(u^{-}(t,x))w(t,x)dxdt,\ \ \forall w\in E_{\Psi },
\end{equation*}%
for every $k\in \mathbb{N},$ where $\Psi $ is the complementary $N$-function
to $\Phi $ and $E_{\widetilde{F}}$ , $E_{\Psi }$ are the closures of the set
of bounded functions in the spaces $L_{\widetilde{F}}^{\ast }((0,T)\times
B(0,k))$ and $L_{\Psi }^{\ast }((0,T)\times B(0,k))$, respectively. The last
two approximations together with (3.6)$_{1}$ yield that
\begin{equation*}
\int\limits_{0}^{T}\int\limits_{B(0,k)\text{ }}f(u_{m}^{+}(t,x))v(t,x)dxdt%
\rightarrow \int\limits_{0}^{T}\int\limits_{B(0,k)\text{ }%
}f(u^{+}(t,x))v(t,x)dxdt,\text{ \ \ }\forall \text{ }v\in L^{\infty
}((0,T)\times R^{n}),
\end{equation*}%
\begin{equation*}
\int\limits_{0}^{T}\int\limits_{B(0,k)\text{ }}f(u_{m}^{-}(t,x))w(t,x)dxdt%
\rightarrow \int\limits_{0}^{T}\int\limits_{B(0,k)\text{ }%
}f(u^{-}(t,x))w(t,x)dxdt,\ \ \forall w\in L^{\infty }((0,T)\times R^{n}).
\end{equation*}%
Now, since
\begin{equation*}
f(u_{m}(t,x))=f(u_{m}^{+}(t,x))+f(u_{m}^{-}(t,x)),
\end{equation*}%
we obtain%
\begin{equation}
\int\limits_{0}^{T}\int\limits_{B(0,k)\text{ }}f(u_{m}(t,x))v(t,x)dxdt%
\rightarrow \int\limits_{0}^{T}\int\limits_{B(0,k)\text{ }%
}f(u(t,x))v(t,x)dxdt,\text{ }\forall \text{ }v\in L^{\infty }((0,T)\times
R^{n}),  \tag{3.13}
\end{equation}%
for every $k\in \mathbb{N}$. By (3.4)-(3.6), we have%
\begin{equation}
\left\langle u_{m},e_{j}\right\rangle \rightarrow \left\langle
u,e_{j}\right\rangle \text{ weakly star in\ \ }L^{\infty }(0,T),  \tag{3.14}
\end{equation}%
\begin{equation*}
\left\langle u_{mt},e_{j}\right\rangle \rightarrow \left\langle
u_{t},e_{j}\right\rangle \text{ weakly\ in}\ L^{2}(\varepsilon ,T),\text{ }%
\forall \varepsilon \in (0,T),
\end{equation*}%
\begin{equation}
\left\langle Au_{m},e_{j}\right\rangle \rightarrow \left\langle \chi
,e_{j}\right\rangle \text{ weakly\ in \ }L^{\frac{p}{p-1}}(0,T),\
\tag{3.15}
\end{equation}%
\begin{equation*}
\left\langle f(u_{m}),e_{j}\right\rangle \rightarrow \left\langle
f(u),e_{j}\right\rangle \text{ in\ }D^{\prime }(0,T).
\end{equation*}%
As a result, we can write that
\begin{equation}
\left\langle u_{t},e_{j}\right\rangle =-\left\langle \chi
,e_{j}\right\rangle -\left\langle f(u),e_{j}\right\rangle +\left\langle
g,e_{j}\right\rangle \text{ in }D^{\prime }(0,T).  \tag{3.16}
\end{equation}%
On the other hand by (3.4), we have
\begin{equation}
\int\limits_{0}^{T}\int\limits_{R^{n}}f(u_{m}(t,x))u_{m}(t,x))dxdt\leq c_{2}.
\tag{3.17}
\end{equation}%
Taking into account (3.8) and applying Fatou's lemma to (3.17), we obtain%
\begin{equation*}
\int\limits_{0}^{T}\int\limits_{R^{n}}f(u(t,x))u(t,x))dxdt\leq c_{2}.
\end{equation*}%
As it was mentioned in the Remark 2.1, the last inequality gives us that
\begin{equation*}
f(u)\in L^{1}(0,T;L^{1}(R^{n})+L^{2}(R^{n})).
\end{equation*}%
So, the equality (3.16) is satisfied a.e. in $(0,T)$ and by the density of $%
\{e_{j}\}_{j=1}^{\infty }$ in $W_{b}$ we get
\begin{equation*}
\left\langle u_{t},v\right\rangle =-\left\langle \chi ,v\right\rangle
-\left\langle f(u),v\right\rangle +\left\langle g,v\right\rangle \text{ a.e.
on }(0,T),
\end{equation*}%
for every $v\in W_{b}$, which together with Lemma 2.2 and Remark 2.2 , gives%
\begin{equation*}
\left\langle u_{t},v\right\rangle =-\left\langle \chi ,v\right\rangle
-\left\langle f(u),v\right\rangle +\left\langle g,v\right\rangle \text{ a.e.
on }(0,T),
\end{equation*}%
for every $v\in W\cap L^{\infty }(R^{n})$. From the last equality it follows
that $u_{t}\in L^{1}(0;T;L^{1}(R^{n})+W^{\ast })$ and
\begin{equation}
u_{t}=-\chi -f(u)+g,\text{ \ in \ }L^{1}(0;T;L^{1}(R^{n})+W^{\ast }).
\tag{3.18}
\end{equation}%
By $u\in L^{\infty }(0,T;L^{2}(R^{n}))$ and $u_{t}\in
L^{1}(0;T;L^{1}(R^{n})+W^{\ast })$, we have
\begin{equation*}
u\in C([0,T];L^{1}(R^{n})+W^{\ast })
\end{equation*}%
and consequently%
\begin{equation*}
u\in C_{s}(0,T;L^{1}(R^{n})+W^{\ast })
\end{equation*}%
Since by \cite[Lemma 8.1, p. 275]{23},
\begin{equation*}
L^{\infty }(0,T;L^{2}(R^{n}))\cap C_{s}(0,T;L^{1}(R^{n})+W^{\ast
})=C_{s}(0,T;L^{2}(R^{n}))\mathbf{,}
\end{equation*}%
we get
\begin{equation}
u\in C_{s}(0,T;L^{2}(R^{n})).  \tag{3.19}
\end{equation}%
Also, applying the argument done in the Remark 2.1 to (3.17) it is easy to
see that the sequence $\left\{ f(u_{m})\right\} _{m=1}^{\infty }$ is bounded
in $L^{1}(0,T;L^{1}(R^{n})+L^{2}(R^{n}))$, which implies the boundedness of $%
\left\{ \left\langle f(u_{m}),e_{j}\right\rangle \right\} _{m=1}^{\infty }$
in $L^{1}(0,T)$. So, taking into account the boundedness of $\left\{
\left\langle A(u_{m}),e_{j}\right\rangle \right\} _{m=1}^{\infty }$ in $L^{%
\frac{p}{p-1}}(0,T)$ and $\left\{ \left\langle f(u_{m}),e_{j}\right\rangle
\right\} _{m=1}^{\infty }$ in $L^{1}(0,T)$, by (3.1), we have that the
sequence\ $\left\{ \left\langle u_{mt},e_{j}\right\rangle \right\}
_{m=1}^{\infty }$ is bounded in $L^{1}(0,T)$ which, together with (3.14),
gives us%
\begin{equation*}
\left\langle u_{m}(0),e_{j}\right\rangle \rightarrow \left\langle
u(0),e_{j}\right\rangle ,\text{ \ \ }j=1,2,...\text{ \ .}
\end{equation*}%
On the other hand, since
\begin{equation*}
u_{m}(0)\rightarrow u_{0}\text{\ \ strongly in\ }L^{2}(R^{n}),
\end{equation*}%
we have $u(0)=u_{0}$. Hence, taking into account (1.3)-(1.5), (3.6) and
passing to the limit in (3.2) when $m\rightarrow \infty $, we get%
\begin{equation*}
\left\Vert u(t)\right\Vert _{L^{2}(R^{n})}^{2}\leq
2\int\limits_{0}^{t}\int\limits_{R^{n}}g(x)u(\tau ,x)dxd\tau +\left\Vert
u(0)\right\Vert _{L^{2}(R^{n})}^{2},
\end{equation*}%
and consequently we have%
\begin{equation*}
\underset{t\rightarrow 0}{\lim \sup }\left\Vert u(t)\right\Vert
_{L^{2}(R^{n})}^{2}\leq \left\Vert u(0)\right\Vert _{L^{2}(R^{n})}^{2}.
\end{equation*}%
By (3.7), (3.19) and the last inequality we obtain%
\begin{equation*}
u\in C([0,T];L^{2}(R^{n})).
\end{equation*}%
Now, since the operator $A:L^{p}(0,T;W)\rightarrow L^{\frac{p}{p-1}%
}(0,T;W^{\ast })$ is bounded, monotone and hemicontinuous, to prove that $\
\chi =Au$, in addition to (3.6)$_{3}$ and (3.6)$_{4}$ we need to show that%
\newline
$\underset{m\rightarrow \infty }{\lim \sup }$ $\int\limits_{0}^{T}\left%
\langle Au_{m}(t),u_{m}(t)\right\rangle dt\leq
\int\limits_{0}^{T}\left\langle \chi (t),u(t)\right\rangle dt$ (see \cite[%
Lemma 2.1, p. 38]{20}). By (3.2), we have%
\begin{equation*}
\int\limits_{0}^{T}\left\langle Au_{m}(t),u_{m}(t)\right\rangle
dt=\int\limits_{0}^{T}\int\limits_{R^{n}}\left( \sigma (x)\left\vert \nabla
u_{m}(\tau ,x)\right\vert ^{p}+\beta (x)\left\vert u_{m}(\tau ,x)\right\vert
^{2}\right) dxd\tau
\end{equation*}%
\begin{equation*}
=\int\limits_{0}^{T}\int\limits_{R^{n}}\left( g(x)u_{m}(\tau
,x)-f(u_{m}(\tau ,x))u_{m}(\tau ,x)\right) dxd\tau
\end{equation*}%
\begin{equation*}
+\frac{1}{2}\left\Vert u_{m}(0)\right\Vert _{L^{2}(R^{n})}^{2}-\frac{1}{2}%
\left\Vert u_{m}(T)\right\Vert _{L^{2}(R^{n})}^{2}.
\end{equation*}%
Since $u_{m}(0)\rightarrow u_{0}$ in $L^{2}(R^{n})$, taking into account
(3.6) and (3.8), and applying Fatou's lemma, we find%
\begin{equation*}
\underset{m\rightarrow \infty }{\lim \sup }\int\limits_{0}^{T}\left\langle
Au_{m},u_{m}\right\rangle dt\leq
\int\limits_{0}^{T}\int\limits_{R^{n}}g(x)u(\tau ,x)dxd\tau
-\int\limits_{0}^{T}\int\limits_{R^{n}}f(u(\tau ,x))u(\tau ,x)dxd\tau
\end{equation*}%
\begin{equation}
+\frac{1}{2}\left\Vert u(0)\right\Vert _{L^{2}(R^{n})}^{2}-\frac{1}{2}%
\left\Vert u(T)\right\Vert _{L^{2}(R^{n})}^{2}.  \tag{3.20}
\end{equation}%
On the other hand, by Remark 2.1 and Lemma 2.3, we can test (3.18) by $%
B_{k}(u)$ on $(\varepsilon ,T)\times R^{n}$, which gives us%
\begin{equation*}
\int\limits_{\varepsilon }^{T}\left\langle \chi (t),B_{k}(u)(t)\right\rangle
dt=\int\limits_{\varepsilon
}^{T}\int\limits_{R^{n}}g(x)B_{k}(u)(t,x)dxdt-\int\limits_{\varepsilon
}^{T}\int\limits_{R^{n}}f(u(x,t))B_{k}(u)(t,x)dxdt
\end{equation*}%
\begin{equation*}
+\int\limits_{R^{n}}u(\varepsilon ,x)B_{k}(u)(\varepsilon
,x)dx-\int\limits_{R^{n}}u(T,x)B_{k}(u)(T,x)dx
\end{equation*}%
\begin{equation*}
+\frac{1}{2}\left\Vert B_{k}(u)(T)\right\Vert _{L^{2}(R^{n})}^{2}-\frac{1}{2}%
\left\Vert B_{k}(u)(\varepsilon )\right\Vert _{L^{2}(R^{n})}^{2}.
\end{equation*}%
Taking into account Lemma 2.3 and passing to the limit as $k\rightarrow
\infty ,$ in the last equality, we obtain%
\begin{equation*}
\int\limits_{\varepsilon }^{T}\left\langle \chi (t),u(t)\right\rangle
dt=\int\limits_{\varepsilon }^{T}\int\limits_{R^{n}}g(x)u(t,x)dxdt-\underset{%
k\rightarrow \infty }{\lim }\int\limits_{\varepsilon
}^{T}\int\limits_{R^{n}}f(u(t,x))B_{k}(u)(t,x)dxdt
\end{equation*}%
\begin{equation}
+\frac{1}{2}\left\Vert u(\varepsilon )\right\Vert _{L^{2}(R^{n})}^{2}-\frac{1%
}{2}\left\Vert u(T)\right\Vert _{L^{2}(R^{n})}^{2}.  \tag{3.21}
\end{equation}%
Since, the sequence $\left\{ f(u(t,x))B_{k}(u)(t,x)\right\} _{k=1}^{\infty }$
is non-decreasing and $B_{k}(u)\rightarrow u$ in $C([0,T];L^{2}(R^{n}))$, by
monotone convergence theorem, we have%
\begin{equation*}
\underset{\mathbf{k\rightarrow \infty }}{\lim }\int\limits_{\varepsilon
}^{T}\int\limits_{R^{n}}f(u(t,x))B_{k}(u)(t,x)dxdt=\int\limits_{\varepsilon
}^{T}\int\limits_{R^{n}}f(u(t,x))u(t,x)dxdt,
\end{equation*}%
which together with (3.21) yields
\begin{equation*}
\int\limits_{\varepsilon }^{T}\left\langle \chi (t),u(t)\right\rangle
dt=\int\limits_{\varepsilon
}^{T}\int\limits_{R^{n}}g(x)u(t,x)dxdt-\int\limits_{\varepsilon
}^{T}\int\limits_{R^{n}}f(u(t,x))u(t,x)dxdt
\end{equation*}%
\begin{equation*}
+\frac{1}{2}\left\Vert u(\varepsilon )\right\Vert _{L^{2}(R^{n})}^{2}-\frac{1%
}{2}\left\Vert u(T)\right\Vert _{L^{2}(R^{n})}^{2}.
\end{equation*}%
Since $u$ $\in $ $C([0,T];L^{2}(R^{n}))$, passing to the limit in the last
equality as $\varepsilon \rightarrow 0$, we get%
\begin{equation*}
\int\limits_{0}^{T}\left\langle \chi (t),u(t)\right\rangle
dt=\int\limits_{0}^{T}\int\limits_{R^{n}}g(x)u(t,x)dxdt-\int\limits_{0}^{T}%
\int\limits_{R^{n}}f(u(t,x))u(t,x)dxdt
\end{equation*}%
\begin{equation*}
+\frac{1}{2}\left\Vert u(0)\right\Vert _{L^{2}(R^{n})}^{2}-\frac{1}{2}%
\left\Vert u(T)\right\Vert _{L^{2}(R^{n})}^{2}.
\end{equation*}%
Taking into account the last equality in (3.20), we obtain $\chi =Au$, which
completes the proof of the existence of the solution.
\end{proof}

\begin{theorem}
Let $u$ and $v$ be weak solutions of problem (1.1)-(1.2), with initial data $%
u_{0}$ and $v_{0}$, respectively. Then
\begin{equation}
\left\Vert u(T)-v(T)\right\Vert _{L^{2}(R^{n})}\leq e^{cT}\left\Vert
u_{0}-v_{0}\right\Vert _{L^{2}(R^{n})}\text{, \ }\forall T\geq 0,  \tag{3.22}
\end{equation}%
where $c$ is the same constant in (1.5).
\end{theorem}

\begin{proof}
Denoting $w=u-v$, we have%
\begin{equation}
\left\{
\begin{array}{c}
w_{t}+\left( A(v+w)-Av\right) -cw+\widetilde{f}(v+w)-\widetilde{f}(v)=0, \\
w(0)=u_{0}-v_{0}.%
\end{array}%
\right.  \tag{3.23}
\end{equation}%
Testing (3.23)$_{1}$ by $B_{k}(w)$ on $(\varepsilon ,T)\times R^{n}$ and
taking into account the monotonicity of the function $\widetilde{f}$, we get%
\begin{equation*}
\int\limits_{R^{n}}w(T,x)B_{k}(w)(T,x)dx-\frac{1}{2}\left\Vert
B_{k}(w)(T)\right\Vert _{L^{2}(R^{n})}^{2}
\end{equation*}%
\begin{equation*}
+\int\limits_{\varepsilon }^{T}\int\limits_{R^{n}}\sigma (x)(\left\vert
\nabla u(t,x)\right\vert ^{p-2}\nabla u(t,x)-\left\vert \nabla
v(t,x)\right\vert ^{p-2}\nabla v(t,x))\cdot \nabla B_{k}(w)(t,x)dxdt
\end{equation*}%
\begin{equation*}
\leq \int\limits_{R^{n}}w(x,\varepsilon )B_{k}(w)(\varepsilon ,x)dx-\frac{1}{%
2}\left\Vert B_{k}(w)(\varepsilon )\right\Vert
_{L^{2}(R^{n})}^{2}+c\int\limits_{\varepsilon
}^{T}\int\limits_{R^{n}}w(x,t)B_{k}(w)(t,x)dxdt.
\end{equation*}%
By the definition of $B_{k}(\cdot )$\ and monotonicity of the function $%
s^{p-1}$\ for $s\geq 0$, we have%
\begin{equation*}
\int\limits_{\varepsilon }^{T}\int\limits_{R^{n}}\sigma (x)(\left\vert
\nabla u(t,x)\right\vert ^{p-2}\nabla u(t,x)\mathbf{-}\left\vert \nabla
v(t,x)\right\vert ^{p-2}\nabla v(t,x)\mathbf{)\cdot }\nabla B_{k}(w)\mathbf{(%
}t,x)dxdt
\end{equation*}%
\begin{equation*}
\mathbf{=}\int\limits_{\varepsilon }^{T}\int\limits_{\left\{ x\in
R^{n},\left\vert w(t,x)\right\vert \leq k\right\} }\sigma (x)(\left\vert
\nabla u(t,x)\right\vert ^{p-2}\nabla u(t,x)\mathbf{-}\left\vert \nabla
v(t,x)\right\vert ^{p-2}\nabla v(t,x)\mathbf{)}\cdot \nabla
(u(t,x)-v(t,x))dxdt
\end{equation*}%
\begin{equation*}
\mathbf{\geq }\int\limits_{\varepsilon }^{T}\int\limits_{\left\{ x\in
R^{n},\left\vert w(t,x)\right\vert \leq k\right\} }\sigma (x)\mathbf{(}%
\left\vert \nabla u(t,x)\right\vert ^{p}\mathbf{-}\left\vert \nabla
u(t,x)\right\vert ^{p-1}\left\vert \nabla v(t,x)\right\vert \mathbf{)}dxdt
\end{equation*}%
\begin{equation*}
\mathbf{+}\int\limits_{\varepsilon }^{T}\int\limits_{\left\{ x\in
R^{n},\left\vert w(t,x)\right\vert \leq k\right\} }\sigma (x)\mathbf{(}%
\left\vert \nabla v(t,x)\right\vert ^{p}\mathbf{-}\left\vert \nabla
v(t,x)\right\vert ^{p-1}\left\vert \nabla u(t,x)\right\vert \mathbf{)}dxdt
\end{equation*}%
\begin{equation*}
\mathbf{=}\int\limits_{\varepsilon }^{T}\int\limits_{\left\{ x\in
R^{n},\left\vert w(t,x)\right\vert \leq k\right\} }\sigma (x)\mathbf{(}%
\left\vert \nabla u(t,x)\right\vert ^{p-1}\mathbf{-}\left\vert \nabla
v(t,x)\right\vert ^{p-1}\mathbf{)(}\left\vert \nabla u(t,x)\right\vert
\mathbf{-}\left\vert \nabla v(t,x)\right\vert \mathbf{)}dxdt\geq 0.
\end{equation*}%
By the last two inequalities, we find%
\begin{equation*}
\int\limits_{R^{n}}w(x,T)B_{k}(w)(T,x)dx-\frac{1}{2}\left\Vert
B_{k}(w)(T)\right\Vert _{L^{2}(R^{n})}^{2}\leq
\int\limits_{R^{n}}w(\varepsilon ,x)B_{k}(w)(\varepsilon ,x)dx
\end{equation*}%
\begin{equation*}
-\frac{1}{2}\left\Vert B_{k}(w)(\varepsilon )\right\Vert
_{L^{2}(R^{n})}^{2}+c\int\limits_{\varepsilon
}^{T}\int\limits_{R^{n}}w(x,t)B_{k}(w)(t,x)dxdt
\end{equation*}%
Passing to the limit as $k\rightarrow \infty $ and $\varepsilon \rightarrow
0 $ in the above inequality and taking into account (3.23)$_{2}$, we obtain%
\begin{equation*}
\frac{1}{2}\left\Vert w(T)\right\Vert _{L^{2}(R^{n})}^{2}\leq \frac{1}{2}%
\left\Vert u_{0}-v_{0}\right\Vert
_{L^{2}(R^{n})}^{2}+c\int\limits_{0}^{T}\left\Vert w(t)\right\Vert
_{L^{2}(R^{n})}^{2}dt,\text{ \ }\forall T\geq 0,
\end{equation*}%
which by Gronwall's lemma yields (3.22).
\end{proof}

Thus, by Theorem 3.1 and Theorem 3.2, under the conditions (1.3)-(1.5) the
solution operator $S(t)u_{0}=u(t)$ of problem (1.1)-(1.2) generates a
strongly continuous$\ $semigroup in $L^{2}(R^{n})$.

\section{Existence of the global attractor}

We begin with the existence of the absorbing set for the semigroup $\left\{
S(t)\right\} _{t\geq 0}.$

\begin{theorem}
Assume that the conditions (1.3)-(1.5) are satisfied. Then the semigroup $%
\left\{ S(t)\right\} _{t\geq 0}$ has a bounded absorbing set in \ $%
L^{2}(R^{n})$, that is, there is a bounded set $B_{0}$ in \ $L^{2}(R^{n})$
such that for any bounded subset $B$ of $L^{2}(R^{n}),$ there exists a $%
T_{0}=T_{0}(B)>0$ such that $\ S(t)B\subset B_{0}$ for every $t\geq T_{0}$.
\end{theorem}

\begin{proof}
Multiplying the equation (3.1)$_{j}$ by the function $c_{mj}(t)$, for each $%
j $, adding these relations for \ $j=1,...,m$, we get the following equality
:%
\begin{equation*}
\frac{1}{2}\frac{d}{dt}\left\Vert u_{m}(t)\right\Vert
_{L^{2}(R^{n})}^{2}+\int\limits_{R^{n}}\sigma (x)\left\vert \nabla
u_{m}(t,x)\right\vert ^{p}dx+\int\limits_{R^{n}}\beta (x)\left\vert
u_{m}(t,x)\right\vert ^{2}dx
\end{equation*}%
\begin{equation}
+\int\limits_{R^{n}}f(u_{m}(t,x))u_{m}(t,x)dx=\int%
\limits_{R^{n}}g(x)u_{m}(t,x)dx,\ \ \ \text{\ }\forall t\geq 0\text{.}
\tag{4.1}
\end{equation}%
Since
\begin{equation*}
1+\int\limits_{R^{n}}\sigma (x)\left\vert \nabla u_{m}(t,x)\right\vert
^{p}dx\geq \left( \int\limits_{R^{n}}\sigma (x)\left\vert \nabla
u_{m}(t,x)\right\vert ^{p}dx\right) ^{\frac{2}{p}},
\end{equation*}%
by taking into account Lemma 2.2 in (4.1), we obtain
\begin{equation*}
\frac{d}{dt}\left\Vert u_{m}(t)\right\Vert
_{L^{2}(R^{n})}^{2}+c_{1}\left\Vert u_{m}(t)\right\Vert
_{L^{2}(R^{n})}^{2}\leq c_{2}\left\Vert g\right\Vert _{L^{2}(R^{n})}^{2}+2,
\end{equation*}%
and consequently%
\begin{equation*}
\left\Vert u_{m}(t)\right\Vert _{L^{2}(R^{n})}^{2}\leq e^{-c_{1}t}\left\Vert
u_{m}(0)\right\Vert _{L^{2}(R^{n})}^{2}+\frac{1}{c_{1}}(c_{2}\left\Vert
g\right\Vert _{L^{2}(R^{n})}^{2}+2),
\end{equation*}%
for some $c_{1}>0$ and $c_{2}>0.$\ By (3.6)$_{1}$\ and (3.6)$_{2}$, we have%
\textbf{\ }%
\begin{equation*}
u_{m}\rightarrow u\text{ weakly in }C([\varepsilon ,T];L^{2}(R^{n}))\text{,
\ }\forall T>\varepsilon >0\text{,}
\end{equation*}%
which yields%
\begin{equation*}
u_{m}(t)\rightarrow u(t)\text{ weakly in }L^{2}(R^{n})\text{, \ }\forall t>0%
\text{.}
\end{equation*}

On the other hand, since $u_{m}(0)\rightarrow u_{0}$\ strongly in $%
L^{2}(R^{n})$, passing to the limit in the last inequality we find that
\begin{equation*}
B_{0}=\left\{ u\in L^{2}(R^{n}):\text{ }\left\Vert u\right\Vert
_{L^{2}(R^{n})}\leq \frac{1}{c_{1}}(c_{2}\left\Vert g\right\Vert
_{L^{2}(R^{n})}^{2}+2)+1\right\}
\end{equation*}%
is an absorbing set for $\left\{ S(t)\right\} _{t\geq 0}$.
\end{proof}

Now, let's prove the asymptotic compactness property of the semigroup $%
\left\{ S(t)\right\} _{t\geq 0}.$

\begin{theorem}
Let conditions (1.3)-(1.5) hold and $B$ be a bounded subset of $L^{2}(R^{n})$%
. Then the set $\left\{ S(t_{m})\varphi _{m}\right\} _{m=1}^{\infty }$ is
relatively compact in $L^{2}(R^{n}),$ where $t_{m}\rightarrow \infty $, $%
\left\{ \varphi _{m}\right\} _{m=1}^{\infty }\subset B$.
\end{theorem}

\begin{proof}
By Theorem 4.1, there exists $T_{0}=T_{0}(B)>0$ such that
\begin{equation}
\left\Vert S(t)\varphi \right\Vert _{L^{2}(R^{n})}\leq c_{1},\hspace{1cm}%
\forall \hspace{0.3cm}t\geq T_{0},\text{ \ }\forall \varphi \in B.  \tag{4.2}
\end{equation}%
For any $T>0$ and $\left\{ t_{m_{k}}\right\} \subset \left\{ t_{m}\right\} $
such that $t_{m_{k}}\geq T+T_{0}$ let us define%
\begin{equation}
u_{k}(t):=S(t+t_{m_{k}}-T)\varphi _{m_{k}},  \tag{4.3}
\end{equation}%
where $u_{k}$ is the solution of (1.1) with the initial condition $%
u_{k}(0)=S(t_{m_{k}}-T)\varphi _{m_{k}}.$ Putting $u_{k}$ instead of $u$ in
(1.1), formally multiplying the obtained equation by $u_{k}$ and $tu_{kt}$,
integrating over $(0,T)\times R^{n}$ and then taking into account condition
(1.3) -(1.5) and (4.2)-(4.3), we find the following estimates%
\begin{equation*}
\left\{
\begin{array}{c}
\left\Vert u_{k}\right\Vert _{L^{p}(0,T;W)}+\left\Vert u_{k}\right\Vert
_{L^{\infty }(0,T;L^{2}(R^{n}))}\text{ \ \ \ \ \ \ \ \ \ \ \ \ \ \ \ \ \ \ \
\ } \\
+\int\limits_{0}^{T}\int\limits_{R^{n}}f(u_{k}(t,x))u_{k}(t,x)dxdt\leq c_{2},%
\text{ \ \ \ \ \ \ \ \ \ \ \ \ \ \ \ \ \ \ \ } \\
\left\Vert u_{k}\right\Vert _{L^{p}(0,T;W^{1,\frac{2n}{n+2}}(B(0,r))}\text{
\ \ \ \ \ \ \ \ \ \ \ \ \ \ \ \ \ \ \ \ \ \ \ \ \ \ \ \ \ \ \ \ \ } \\
+\left\Vert u_{kt}\right\Vert _{L^{2}(\varepsilon ,T;L^{2}(R^{n}))}\leq
c_{\varepsilon ,r},\text{ \ \ }\forall \varepsilon \in (0,T),\text{ \ }%
\forall r>0.%
\end{array}%
\right.
\end{equation*}%
These estimates can be justified by using Galerkin's approximation as it was
done in the previous section. So, repeating the argument done in the proof
of Theorem 3.1, for the subsequence of $u_{k}$, without changing the name of
it, we have
\begin{equation}
\left\{
\begin{array}{c}
u_{k}\rightarrow w\text{ \ weakly in \ }L^{p}(0,T;W),\text{ \ \ \ \ \ \ \ \
\ \ \ } \\
u_{k}\rightarrow w\text{ \ weakly star in}\hspace{0.2cm}L^{\infty
}(0,T;L^{2}(R^{n})) \\
u_{kt}\rightarrow w_{t}\text{ \ weakly in }L^{2}(\varepsilon
,T;L^{2}(R^{n})),\text{ \ \ \ } \\
Au_{k}\rightarrow \chi \text{ weakly\ in }\ L^{\frac{p}{p-1}}(0,T;W^{\ast }),%
\text{ \ \ \ \ \ } \\
u_{k}\rightarrow w\text{ \ a.e.\hspace{0.2cm}in}\hspace{0.2cm}(0,T)\times
R^{n},\text{ \ \ \ \ \ \ \ \ \ \ \ \ \ \ \ } \\
f(u_{k})\rightarrow f(w)\text{ in}\ \ D^{\prime }(0,T\times R^{n}),\text{ \
\ \ \ \ \ \ \ \ \ }%
\end{array}%
\right.  \tag{4.4}
\end{equation}%
where $\chi $ $\in L^{\frac{p}{p-1}}(0,T;W^{\ast })$, $w\in $ $L^{\infty
}(0,T;L^{2}(R^{n}))\cap L^{p}(0,T;W)\cap W^{1,2}(\varepsilon
,T;L^{2}(R^{n})) $ and\newline
$\int\limits_{0}^{T}\int\limits_{R^{n}}f(w(t,x))w(t,x)dxdt<\infty $. Now,
putting $u_{k}$ instead of $u$ in (1.1) and passing to the limit, we find%
\begin{equation*}
w_{t}+\chi +f(w)=g.
\end{equation*}%
Taking into account Remark 2.1 and Lemma 2.3, and testing the above equation
by $B_{k}(w)$ on $(s,T)\times R^{n}$, we obtain%
\begin{equation*}
\int\limits_{s}^{T}\left\langle \chi (t),B_{k}(w)(t)\right\rangle
dt=\int\limits_{s}^{T}\int\limits_{R^{n}}g(x)B_{k}(w)(t,x)dxdt-\int%
\limits_{s}^{T}\int\limits_{R^{n}}f(w(t,x))B_{k}(w)(t,x)dxdt
\end{equation*}%
\begin{equation*}
+\int\limits_{R^{n}}w(s,x)B_{k}(w)(s,x)dx-\int%
\limits_{R^{n}}w(T,x)B_{k}(w)(T,x)dx
\end{equation*}%
\begin{equation*}
+\frac{1}{2}\left\Vert B_{k}(w)(T)\right\Vert _{L^{2}(R^{n})}^{2}-\frac{1}{2}%
\left\Vert B_{k}(w)(s)\right\Vert _{L^{2}(R^{n})}^{2},\text{ \ }\forall
T\geq s>0\text{.}
\end{equation*}%
Again repeating the argument done in the proof of Theorem 3.1, passing to
the limit as $k\rightarrow \infty $ and integrating the obtained equality
from $\varepsilon $ to $T$ \ with respect to $s$, we get%
\begin{equation*}
\frac{1}{2}(T-\varepsilon )\left\Vert w(T)\right\Vert
_{L^{2}(R^{n})}^{2}+\int\limits_{\varepsilon
}^{T}\int\limits_{s}^{T}\left\langle \chi (t),w(t)\right\rangle
dtds+\int\limits_{\varepsilon
}^{T}\int\limits_{s}^{T}\int\limits_{R^{n}}f(w(x,t))w(x,t)dxdtds
\end{equation*}%
\begin{equation}
=\frac{1}{2}\int\limits_{\varepsilon }^{T}\left\Vert w(s)\right\Vert
_{L^{2}(R^{n})}^{2}ds+\int\limits_{\varepsilon
}^{T}\int\limits_{s}^{T}\int\limits_{R^{n}}g(x)w(x,t)dxdtds,\text{ \ }%
\forall T\geq \varepsilon >0.  \tag{4.5}
\end{equation}%
Now, putting $u_{k}$ instead of $u$ in (1.1), testing this equation by $%
B_{m}(u_{k})$ on $(s,T)\times R^{n}$, integrating the obtained equality from
$\varepsilon $ to $T$ \ with respect to $s$ and passing to the limit as $%
m\rightarrow \infty $, we obtain
\begin{equation*}
\frac{1}{2}(T-\varepsilon )\left\Vert u_{k}(T)\right\Vert
_{L^{2}(R^{n})}^{2}+\int\limits_{\varepsilon
}^{T}\int\limits_{s}^{T}\left\langle Au_{k}(t),u_{k}(t)\right\rangle
dtds+\int\limits_{\varepsilon
}^{T}\int\limits_{s}^{T}\int\limits_{R^{n}}f(u_{k}(t,x))u_{k}(t,x)dxdtds
\end{equation*}%
\begin{equation}
=\frac{1}{2}\int\limits_{\varepsilon }^{T}\left\Vert u_{k}(s)\right\Vert
_{L^{2}(R^{n})}^{2}ds+\int\limits_{\varepsilon
}^{T}\int\limits_{s}^{T}\int\limits_{R^{n}}g(x)u_{k}(t,x)dxdtds,\text{ \ }%
\forall T\geq \varepsilon >0.  \tag{4.6}
\end{equation}%
By (4.2)-(4.4), we have
\begin{equation*}
\underset{k\rightarrow \infty }{\lim \inf }\int\limits_{\varepsilon
}^{T}\int\limits_{s}^{T}\left\langle Au_{k}(t),u_{k}(t)\right\rangle
dtds-\int\limits_{\varepsilon }^{T}\int\limits_{s}^{T}\left\langle \chi
(t),w(t)\right\rangle dtds
\end{equation*}%
\begin{equation*}
=\underset{k\rightarrow \infty }{\lim \inf }\int\limits_{\varepsilon
}^{T}\int\limits_{s}^{T}\left\langle
Au_{k}(t)-Aw(t),u_{k}(t)-w(t)\right\rangle dtds
\end{equation*}%
\begin{equation}
\geq c_{3}\underset{k\rightarrow \infty }{\lim \inf }\int\limits_{%
\varepsilon }^{T}\int\limits_{s}^{T}\left\Vert u_{k}(t)-w(t)\right\Vert
_{W}^{p}dtds,  \tag{4.7}
\end{equation}%
and
\begin{equation}
\underset{k\rightarrow \infty }{\lim \inf }\int\limits_{\varepsilon
}^{T}\int\limits_{s}^{T}\int\limits_{R^{n}}f(u_{k}(t,x))u_{k}(t,x)dxdtds\geq
\int\limits_{\varepsilon
}^{T}\int\limits_{s}^{T}\int\limits_{R^{n}}f(w(t,x))w(t,x)dxdtds.  \tag{4.8}
\end{equation}%
So, by (4.4)-(4.8), we obtain%
\begin{equation*}
\mathbf{(}T-\varepsilon \mathbf{)}\underset{k\rightarrow \infty }{\lim \inf }%
\left\Vert u_{k}(T)-w(T)\right\Vert _{L^{2}(R^{n})}^{2}\mathbf{+}2c_{3}%
\underset{k\rightarrow \infty }{\lim \inf }\int\limits_{\varepsilon }^{T}(s%
\mathbf{-}\varepsilon )\left\Vert u_{k}(s)-w(s)\right\Vert _{W}^{p}ds
\end{equation*}%
\begin{equation*}
\leq \underset{k\rightarrow \infty }{\lim \inf }\int\limits_{\varepsilon
}^{T}\left\Vert u_{k}(s)-w(s)\right\Vert _{L^{2}(R^{n})}^{2}ds\mathbf{,}%
\text{ \ }\mathbf{\forall }T\geq \varepsilon ,
\end{equation*}%
and consequently%
\begin{equation*}
\mathbf{(}T-\varepsilon \mathbf{)}\underset{k\rightarrow \infty }{\lim \inf }%
\left\Vert u_{k}(T)-w(T)\right\Vert _{L^{2}(R^{n})}^{2}\mathbf{+}c_{4}%
\underset{k\rightarrow \infty }{\lim \inf }\int\limits_{2\varepsilon
}^{T}s\left\Vert u_{k}(s)-w(s)\right\Vert _{L^{2}(R^{n})}^{p}ds
\end{equation*}%
\begin{equation*}
\mathbf{\leq }\underset{k\rightarrow \infty }{\lim \inf }\int\limits_{%
\varepsilon }^{T}\left\Vert u_{k}(s)-w(s)\right\Vert _{L^{2}(R^{n})}^{2}ds%
\mathbf{,}\text{ \ }\mathbf{\forall }T\geq 2\varepsilon \mathbf{.}
\end{equation*}%
Taking into account (4.2), (4.3) and (4.4)$_{2}$ in the integral on the
right hand side of the above inequality, we have%
\begin{equation*}
\int\limits_{\varepsilon }^{T}\left\Vert u_{k}(s)-w(s)\right\Vert
_{L^{2}(R^{n})}^{2}ds\mathbf{\leq }2c_{1}^{2}\varepsilon \mathbf{+}%
2c_{1}\int\limits_{2\varepsilon }^{T}\left\Vert u_{k}(s)-w(s)\right\Vert
_{L^{2}(R^{n})}ds,\text{ \ }\forall T\geq 2\varepsilon \mathbf{.}
\end{equation*}%
Applying Holder inequality, we get%
\begin{equation*}
\int\limits_{2\varepsilon }^{T}\left\Vert u_{k}(s)-w(s)\right\Vert
_{L^{2}(R^{n})}ds
\end{equation*}%
\begin{equation*}
\leq \left\{
\begin{array}{c}
\log ^{\frac{1}{2}}(\frac{T}{2\varepsilon })\left( \int\limits_{2\varepsilon
}^{T}s\left\Vert u_{k}(s)-w(s)\right\Vert _{L^{2}(R^{n})}^{p}ds\right) ^{%
\frac{1}{2}},\text{ \ }p=2, \\
\left( \frac{p-1}{p-2}\left( T^{\frac{p-2}{p-1}}-(2\varepsilon )^{\frac{p-2}{%
p-1}}\right) \right) ^{\frac{p-1}{p}}\left( \int\limits_{2\varepsilon
}^{T}s\left\Vert u_{k}(s)-w(s)\right\Vert _{L^{2}(R^{n})}^{p}ds\right) ^{%
\frac{1}{p}},\text{\ }p>2%
\end{array}%
,\right. \text{ \ \ }\forall T\geq 2\varepsilon .
\end{equation*}%
The last three inequalities together with (4.4)$_{2}$-(4.4)$_{3}$ give us%
\begin{equation*}
\underset{k\rightarrow \infty }{\lim \inf }\underset{i\rightarrow \infty }{%
\lim \inf }\left\Vert u_{k}(T)-u_{i}(T)\right\Vert _{L^{2}(R^{n})}^{2}\leq
\frac{2c_{1}^{2}\varepsilon }{T-\varepsilon }
\end{equation*}%
\begin{equation*}
+\frac{c_{5}}{T-\varepsilon }\left\{
\begin{array}{c}
\log (\frac{T}{2\varepsilon }),\text{ \ \ \ \ \ \ \ \ \ \ \ \ \ \ \ \ \ \ \
\ }p=2, \\
\frac{p-1}{p-2}\left( T^{\frac{p-2}{p-1}}-(2\varepsilon )^{\frac{p-2}{p-1}%
}\right) ,\text{ }p>2%
\end{array}%
,\right. \text{ \ \ }\forall T\geq 2\varepsilon
\end{equation*}%
and consequently%
\begin{equation*}
\liminf_{m\rightarrow \infty }\liminf_{k\rightarrow \infty }\left\Vert
S(t_{m})\varphi _{m}-S(t_{k})\varphi _{k}\right\Vert _{L^{2}(R^{n})}^{2}\leq
\frac{2c_{1}^{2}\varepsilon }{T-\varepsilon }
\end{equation*}%
\begin{equation*}
+\frac{c_{5}}{T-\varepsilon }\left\{
\begin{array}{c}
\log (\frac{T}{2\varepsilon }),\text{ \ \ \ \ \ \ \ \ \ \ \ \ \ \ \ \ \ \ \
\ }p=2, \\
\frac{p-1}{p-2}\left( T^{\frac{p-2}{p-1}}-(2\varepsilon )^{\frac{p-2}{p-1}%
}\right) ,\text{ }p>2%
\end{array}%
,\right. \text{ \ \ \ }\forall T\geq 2\varepsilon .
\end{equation*}%
Passing to the limit in the above inequality as $T\rightarrow \infty $, we
have%
\begin{equation*}
\liminf_{m\rightarrow \infty }\liminf_{k\rightarrow \infty }\left\Vert
S(t_{m})\varphi _{m}-S(t_{k})\varphi _{k}\right\Vert _{L^{2}(R^{n})}=0.
\end{equation*}%
By the same way, one can show that%
\begin{equation}
\underset{k\rightarrow \infty }{\lim \inf }\underset{i\rightarrow \infty }{%
\lim \inf }\left\Vert S(t_{m_{i}})\varphi _{m_{i}}-S(t_{m_{k}})\varphi
_{m_{k}}\right\Vert _{L^{2}(R^{n})}=0,  \tag{4.9}
\end{equation}%
for every subsequence $\left\{ m_{k}\right\} _{k=1}^{\infty }$. Now, using
the argument done at the end of the proof of \cite[Lemma 3.4]{19}, let us
show that the sequence $\left\{ S(t_{m})\varphi _{m}\right\} _{m=1}^{\infty
} $\ is relatively compact in $L^{2}(R^{n})$. If not, then there exists $%
\varepsilon _{0}>0$\ such that the set $\left\{ S(t_{m})\varphi _{m}\right\}
_{m=1}^{\infty }$\ has no finite $\varepsilon _{0}$-net in $L^{2}(R^{n})$.
This means that there exists a subsequence $\left\{ m_{k}\right\}
_{k=1}^{\infty }$, such that\textbf{\ }%
\begin{equation*}
\left\Vert S(t_{m_{i}})\varphi _{m_{i}}-S(t_{m_{k}})\varphi
_{m_{k}}\right\Vert _{L^{2}(R^{n})}\geq \varepsilon _{0},\text{ \ }i\neq k.
\end{equation*}%
The last inequality contradicts (4.9).
\end{proof}

Thus, taking into account Theorem 4.1, Theorem 4.2 and applying \cite[%
Theorem 3.1]{24} we have Theorem 1.1.


\end{document}